\documentclass[11pt]{article}
\usepackage{mathbbold}
\usepackage{amsfonts}
\usepackage{mathrsfs}
\usepackage{bbm}
\usepackage{latexsym,amsfonts,amssymb,amsmath,amsthm}
\usepackage{graphicx}

\usepackage{setspace}

\newtheorem{theorem}{Theorem}[section]
\newtheorem{definition}[theorem]{Definition}
\newtheorem{lemma}[theorem]{Lemma}

\newtheorem{proposition}[theorem]{Proposition}
\numberwithin{equation}{section}
\begin{document}

\baselineskip 20pt

\begin{center}

\textbf{\Large Sufficient Criteria for Existence of Pullback
Attractors for Stochastic Lattice Dynamical Systems with
Deterministic Non-autonomous Terms} \footnote{This work has been
partially supported by NSFC Grants 11071199, NSF of Guangxi Grants
2013GXNSFBA019008 and Guangxi Provincial Department of Research
Project Grants 2013YB102.}

\vskip 0.5cm

{\large  Anhui Gu, Yangrong Li}

\vskip 0.3cm

\textit{School of Mathematics and Statistics, Southwest
University, Chongqing 400715, China}\\

\vskip 1cm

\begin{minipage}[c]{15cm}

\noindent \textbf{Abstract}: We consider the pullback attractors for
non-autonomous dynamical systems generated by stochastic lattice
differential equations with non-autonomous deterministic terms. We
first establish a sufficient condition for existence of pullback
attractors of lattice dynamical systems with both non-autonomous
deterministic and random forcing terms. As an application of the
abstract theory, we prove the existence of a unique pullback
attractor for the first-order lattice dynamical systems with both
deterministic non-autonomous forcing terms and multiplicative white
noise. Our results recover many existing ones on the existences of
pullback attractors for lattice dynamical systems with autonomous
terms or white noises.\vspace{5pt}

\vspace{5pt}


\end{minipage}
\end{center}

\vspace{10pt}

\baselineskip 18pt

\section{Introduction}
\label{1} The study of non-autonomous evolution equations has
attracted several interests from both mathematicians and physicists
due to the effects of time-dependent linear/non-linear forces from
natural phenomena are represented by non-autonomous terms in the
associated models. One of the important concepts for describing the
asymptotic behavior of non-autonomous evolution equations is the
pullback attractor, which generalized the notation of global
attractor for non-autonomous dynamical systems \cite{Arnold,
Chueshov, CDF}. The pullback attractors are different from the
uniform attractor (see e.g. \cite{CLR, CV}) in that they employ
techniques of non-autonomous equations more straightly.

Global attractors, uniform attractors and pullback attractors all
play important roles in the fields of asymptotic behavior of
autonomous and non-autonomous infinite dynamical systems \cite{CLR,
CV, Hale, KR, SY, Temam}. Sometimes, the forwards dynamics may be
hard to describe, in this case, there is not even an attracting
trajectory (which would in general be a moving object) that
describes the dynamics. Especially, in the stochastic cases, the
pullback process produces a fixed subset of the phrase space.
Pullback attractors attract all bounded set, then become appropriate
alternatives to study the asymptotic behavior of dynamical systems.

Lattice dynamical systems, which are coupled systems with ODEs on
infinite lattices, have drawn much attention from mathematicians and
physicists recently, due to the wide range of applications in
various areas (e.g. \cite{Chow}). For autonomous deterministic
lattice dynamical systems, we can see e.g. \cite{BLW, Zh1, Zh2, Zh4,
Zh5} for the existence and approximations of attractors. For
non-autonomous deterministic cases, we can see e.g. \cite{Zh8, ZH,
Zh7, Zh9} for the existence and continuity of kernel section,
uniform attractors and pullback exponential attractors. As in the
stochastic cases, stochastic lattice dynamical systems (SLDS) arise
naturally while random influences or uncertainties are taken into
account in lattice dynamical systems, these noises may play an
important role as intrinsic phenomena rather than just compensation
of defects in deterministic models. Since Bates et al. \cite{BLL}
initiated the study of SLDS, lots of work have been done regarding
the existence of global random attractors for SLDS with white
additive/multiplicative noises in regular or weight spaces of
infinite sequences, see e.g. \cite{Wang2, CL, HSZ, Zhou1}. For
lattice dynamical systems perturbed by other ``rough" noises, we can
refer to e.g. \cite{Gu1, Gu2} for more details. As we can seen that
all the systems above are considered with the autonomous
deterministic external forcing terms (if indeed exist!). There is no
results on pullback attractors for general non-autonomous SLDS (with
time-dependent deterministic coefficients and external forcing
terms) as far as we know.

Motivated by \cite{Wang1} and \cite{ZH}, we consider the existence
of non-autonomous dynamical systems generated by lattice
differential equations with both non-autonomous deterministic and
stochastic forcing terms. By borrowing the main framework of
\cite{Wang1} on two parametric space, we first set up the abstract
structure for the continuous cocycle. As a typical example, we
investigate the following stochastic lattice dynamical systems
(SLDS) with time-dependent external forcing terms:
\begin{equation}
\label{lattice1} \frac{du_i(t)}{dt}=\nu_i(t)
(u_{i-1}-2u_i+u_{i+1})-\lambda_i(t) u_i-f_i(u_i,
t)+g_i(t)+u_i\circ\frac{dw(t)}{dt},
\end{equation}
where $i\in \mathbb{Z}$, $\mathbb{Z}$ denotes the integer set;
$u_i\in \mathbb{R}$, $\nu_i(t)$ and $\lambda_i(t)$ are locally
integrable in $t$; $g_i\in C(\mathbb{R}, \mathbb{R})$ and $f_i\in
C(\mathbb{R}\times \mathbb{R}, \mathbb{R})$ satisfies proper
dissipative conditions; $w(t)$ is a Brownian motion (Wiener process)
and $\circ$ denotes the Stratonovich sense of the stochastic term.

Stochastic systems similar to \eqref{lattice1} are discrete of the
\textit{Reaction-Diffusion equation} which used to model the
phenomena of stochastic resonance in biology and physics, where $f$
is a time-dependent input signal and $w$ is a Wiener process used to
test the impact of stochastic fluctuations on $f$. For this topic,
we can see e.g. \cite{GH, Tuckwell, WM, WPP} and the references
therein. The main difference between system \eqref{lattice1} and the
model considered in \cite{BLL} is the coefficients and deterministic
external forcing terms are time-dependent. In this case, the
existing results of one parametric space cannot be applied directly.
We first need to introduce two parametric space to describe the
dynamics of the SLDS: one is responsible for deterministic forcing
and the other is responsible for stochastic perturbations. Then we
applied the skeleton to \eqref{lattice1}.

The outline of the paper is as follows. In the next section, we
recall some results regarding pullback attractor for non-autonomous
dynamical systems over two parametric spaces in \cite{Wang1}. In
section \ref{cocycle}, we establish the conditions on the existence
of pullback attractors for cocycles over two parametric spaces. In
section \ref{lattice equations}, a sufficient condition for the
existence of pullback attractors for lattice differential equations
with both non-autonomous deterministic and random forcing terms is
given. As a example of the result in previous sections, the
existence of pullback attractor for the first-order SLDS with
time-dependent deterministic force and multiplicative white noise is
studied in the last section.

\section{Preliminaries}\label{2}
For the reader's convenience, we recall the theory of pullback
random dynamical systems over two parametric spaces in \cite{Wang1}.

Let $\Omega_1$ be a nonempty  set and $\{\theta_{1,t}\}_{t \in
\mathbb{R}}$  be a family of mappings   from $\Omega_1$ into itself
such that $\theta_{1, 0}  $ is the identity on $\Omega_1$ and
$\theta_{1,  s+t}  = \theta_{1,t}   \theta_{1,s}  $ for all $t, s
\in \mathbb{R}$. Let $(\Omega_2, \mathcal{F}_2, P)$ be a probability
space and $\theta_2 : \mathbb{R}\times \Omega_2 \to \Omega_2$ be  a
$(\mathcal{B} (\mathbb{R}) \times \mathcal{F}_2, \mathcal{F}_2)$
-measurable mapping such that $\theta_2(0,\cdot) $ is the identity
on $\Omega_2$, $\theta_2 (s+t,\cdot) = \theta_2 (t,\cdot) \theta_2
(s,\cdot)$ for all $t, s \in \mathbb{R}$ and $P \theta_2 (t,\cdot)
=P$ for all $t \in \mathbb{R}$. We usually write $\theta_2 (t,
\cdot)$ as $\theta_{2,t}$ and   call both $(\Omega_1,
\{\theta_{1,t}\}_{t \in \mathbb{R}})$ and $(\Omega_2, \mathcal{F}_2,
P,  \{\theta_{2,t}\}_{t \in \mathbb{R}})$ a parametric   dynamical
system.

Let  $(X, d)$   be  a complete separable  metric space with  Borel
$\sigma$-algebra $\mathcal{B} (X)$. Denote by $2^X$   the collection
of all subsets of $X$. A set-valued mapping $K: \Omega_1 \times
\Omega_2 \to  2^X$ is called measurable with respect to
$\mathcal{F}_2$ in $\Omega_2$ if  the value  $K(\omega_1, \omega_2)$
is a closed  nonempty subset of $X$ for all $\omega_1 \in \Omega_1$
and $\omega_2 \in \Omega_2$, and the mapping $ \omega_2 \in \Omega_2
 \to d(x, K(\omega_1, \omega_2) )$
is $(  \mathcal{F}_2, \ \mathcal{B}(\mathbb{R}) )$-measurable for
every  fixed $x \in X$ and $\omega_1 \in \Omega_1$. If $K$ is
measurable  with respect to $\mathcal{F}_2$ in $\Omega_2$,   then we
say that the family $\{K(\omega_1, \omega_2): \omega_1 \in \Omega_1,
\omega_2 \in \Omega_2 \}$ is measurable with respect to
$\mathcal{F}_2$ in $\Omega_2$. We now define a cocycle on $X$ over
two parametric spaces.

\begin{definition}\label{Def1}
Let $(\Omega_1,  \{\theta_{1,t}\}_{t \in \mathbb{R}})$ and
$(\Omega_2, \mathcal{F}_2, P,  \{\theta_{2,t}\}_{t \in \mathbb{R}})$
be parametric  dynamical systems. A mapping $\Phi$: $ \mathbb{R}^+
\times \Omega_1 \times \Omega_2 \times X \to X$ is called a
\textit{continuous  cocycle} on $X$ over $(\Omega_1,
\{\theta_{1,t}\}_{t \in \mathbb{R}})$ and $(\Omega_2, \mathcal{F}_2,
P,  \{\theta_{2,t}\}_{t \in \mathbb{R}})$ if   for all $\omega_1\in
\Omega_1$, $\omega_2 \in \Omega_2 $ and    $t, \tau \in
\mathbb{R}^+$, the following conditions (i)-(iv) are satisfied:
\begin{itemize}
\item [(i)]   $\Phi (\cdot, \omega_1, \cdot, \cdot): \mathbb{R}^+
\times \Omega_2 \times X \to X$ is $(\mathcal{B} (\mathbb{R}^+)
\times \mathcal{F}_2 \times \mathcal{B} (X), \
\mathcal{B}(X))$-measurable;

\item[(ii)]    $\Phi(0, \omega_1, \omega_2, \cdot) $ is the
identity on $X$;

\item[(iii)]    $\Phi(t+\tau, \omega_1, \omega_2, \cdot) =
\Phi(t, \theta_{1,\tau} \omega_1,  \theta_{2,\tau} \omega_2, \cdot)
\Phi(\tau, \omega_1, \omega_2, \cdot)$;

\item[(iv)]    $\Phi(t, \omega_1, \omega_2,  \cdot): X \to  X$
is continuous.
    \end{itemize}
\end{definition}

In the sequel, we use  $\mathcal{D}(X)$ to denote
 a  collection  of  some families of  nonempty subsets of $X$:
\begin{equation*}
\label{defcald} {\mathcal{D}(X)} = \{ D =\{ D(\omega_1, \omega_2 )
\subseteq X: \ D(\omega_1, \omega_2 ) \neq \emptyset,  \ \omega_1
\in \Omega_1, \ \omega_2 \in \Omega_2\} \}.
\end{equation*}

\begin{definition}
Let $\mathcal{D}(X)$ be a collection of some families of nonempty
subsets of $X$ and $K=\{K(\omega_1, \omega_2): \omega_1 \in
\Omega_1, \ \omega_2 \in \Omega_2\} \in \mathcal{D}(X)$. Then $K$ is
called a  \textit{$\mathcal{D}(X)$-pullback absorbing set} for
$\Phi$ if for all $\omega_1 \in \Omega_1$, $\omega_2 \in \Omega_2 $
and for every $B \in \mathcal{D}(X)$, there exists $T= T(B,
\omega_1, \omega_2)>0$ such that
\begin{equation*}
\label{abs1} \Phi(t, \theta_{1,-t} \omega_1, \theta_{2, -t}
\omega_2, B(\theta_{1,-t} \omega_1, \theta_{2,-t} \omega_2  ))
\subseteq  K(\omega_1, \omega_2) \quad \text{for all} \ t \ge T.
\end{equation*}
If, in addition, for all $\omega_1 \in \Omega_1$ and $\omega_2 \in
\Omega_2$, $K(\omega_1, \omega_2)$ is a closed nonempty subset of
$X$ and $K$ is measurable with respect to the $P$-completion of
$\mathcal{F}_2$ in $\Omega_2$, then we say $K$ is a  closed
measurable \textit{$\mathcal{D}(X)$-pullback absorbing  set} for
$\Phi$.
\end{definition}

\begin{definition} \label{asycomp}
Let $\mathcal{D}(X)$ be a collection of  some families of nonempty
subsets of $X$. Then $\Phi$ is said to be
\textit{$\mathcal{D}(X)$-pullback asymptotically compact} in $X$ if
for all $\omega_1 \in \Omega_1$ and $\omega_2 \in \Omega_2$,    the
sequence
\begin{equation*}
\label{asycomp1} \{\Phi(t_n, \theta_{1, -t_n} \omega_1, \theta_{2,
-t_n} \omega_2, x_n)\}_{n=1}^\infty \text{  has a convergent
subsequence  in }   X
\end{equation*}
 whenever
$t_n \to \infty$, and $ x_n\in   B(\theta_{1, -t_n}\omega_1,
\theta_{2, -t_n} \omega_2 )$   with $\{B(\omega_1, \omega_2):
\omega_1 \in \Omega_1, \ \omega_2 \in \Omega_2 \}   \in
\mathcal{D}(X)$.
\end{definition}

\begin{definition} \label{defatt}
Let $\mathcal{D}(X)$ be a collection of some families of nonempty
subsets of $X$ and
 $\mathcal{A} = \{\mathcal{A} (\omega_1, \omega_2): \omega_1 \in \Omega_1,
\omega_2 \in \Omega_2 \} \in \mathcal{D}(X) $. Then $\mathcal{A}$ is
called a    \textit{$\mathcal{D}(X)$-pullback attractor} for $\Phi$
if the following  conditions (i)-(iii) are  fulfilled:
\begin{itemize}
\item [(i)]   $\mathcal{A}$ is measurable
with respect to the $P$-completion of $\mathcal{F}_2$ in $\Omega_2$
and $\mathcal{A}(\omega_1, \omega_2)$ is compact for all $\omega_1
\in \Omega_1$ and    $\omega_2 \in \Omega_2$.

\item[(ii)]   $\mathcal{A}$  is invariant, that is,
for every $\omega_1 \in \Omega_1$ and $\omega_2 \in \Omega_2$,
$$ \Phi(t, \omega_1, \omega_2, \mathcal{A}(\omega_1, \omega_2)   )
= \mathcal{A} (\theta_{1,t} \omega_1, \theta_{2,t} \omega_2 ), \ \
\forall \   t \ge 0.
$$

\item[(iii)]   $\mathcal{A}  $
attracts  every  member   of   $\mathcal{D}(X)$,  that is, for every
$B = \{B(\omega_1, \omega_2): \omega_1 \in \Omega_1, \omega_2 \in
\Omega_2\}
 \in \mathcal{D}(X)$ and for every $\omega_1 \in \Omega_1$ and
$\omega_2 \in \Omega_2$,
$$ \lim_{t \to  \infty} d (\Phi(t, \theta_{1,-t}\omega_1,
\theta_{2,-t}\omega_2, B(\theta_{1,-t}\omega_1,
\theta_{2,-t}\omega_2) ) , \mathcal{A} (\omega_1, \omega_2 ))=0.
$$
 \end{itemize}
\end{definition}

The following result on the  existence and uniqueness of
\textit{$\mathcal{D}(X)$-pullback attractors} for $\Phi$ can be
found in \cite{Wang1}.

 \begin{proposition} \label{att}
Let $\Phi$  be a continuous cocycle on $X$ over $(\Omega_1,
\{\theta_{1,t}\}_{t \in \mathbb{R}})$ and $(\Omega_2, \mathcal{F}_2,
P,  \{\theta_{2,t}\}_{t \in \mathbb{R}})$. Suppose that
$K=\{K(\omega_1, \omega_2): \omega_1 \in \Omega_1, \omega_2 \in
\Omega_2\}\in \mathcal{D}(X)$ is a closed measurable (w.r.t. the
$P$-completion of $\mathcal{F}_2$) \textit{$\mathcal{D}(X)$-pullback
absorbing set} for $\Phi$ in $\mathcal{D}(X)$ and $\Phi$ is
\textit{$\mathcal{D}(X)$-pullback asymptotically compact} in $X$.
Then $\Phi$ has a unique \textit{$\mathcal{D}(X)$-pullback
attractor} $\mathcal{A}=\{\mathcal{A}(\omega_1, \omega_2):  \omega_1
\in \Omega_1, \omega_2 \in \Omega_2\}\in \mathcal{D}(X)$ which is
given by
\begin{align*}\label{attform1}
\mathcal{A} (\omega_1, \omega_2) &=\bigcap_{\tau \ge 0} \overline{
\bigcup_{t\ge \tau} \Phi(t, \theta_{1,-t} \omega_1, \theta_{2, -t}
\omega_2, K(\theta_{1,-t} \omega_1, \theta_{2,-t}\omega_2))}.
 \end{align*}
  \end{proposition}

To describe the size of subsets in a Banach space $X$, we introduce
the concept of \textit{Kolmogorov's $\varepsilon$-entropy}. Let $Y$
be a subset of $X$. Given $\varepsilon>0$, we define

\begin{equation*}
n_{\varepsilon}(Y):=\min\{n\ge 1: Y\subset
\bigcup_{i=1}^n\mathcal{N}(x_i, \varepsilon)\quad \mbox{for some}
\quad x_1, \ldots, x_n\in X \},
\end{equation*}
where $\mathcal{N}(x_i, \varepsilon)=\{y\in X:
\|y-x_i\|_X<\varepsilon\}$. The \textit{Kolmogorov
$\varepsilon$-entropy} of the subset $Y$ of $X$ is the number
\begin{equation}\label{Kol}
\mathbb{K}_{\varepsilon}(Y):=\ln n_{\varepsilon}(Y)\in [0, +\infty].
\end{equation}

\section{Pullback attractors for cocycles in
$\ell^2$}\label{cocycle} In this section, we provide some sufficient
conditions for the existence of pullback attractors for cocycles in
$\ell^2$.

Let $D$ be a bounded nonempty subset of $\ell^2$, denote by
$\|D\|=\sup_{u\in D}\|u\|$. Suppose $D=\{D(\omega_1, \omega_2):
\omega_1\in \Omega_1, \omega_2\in \Omega_2\}$ is a family of bounded
nonempty subsets of $\ell^2$ satisfying, for every $\gamma>0$,
\begin{equation}\label{s1}
\lim_{s\rightarrow+\infty}e^{-\gamma s} \|D(\theta_{1, -s}\omega_1,
\theta_{2, -s}\omega_2)\|^2=0.
\end{equation}
Denote by $\mathcal{D}(\ell^2)$ the collection of all family of
bounded nonempty subsets of $\ell^2$,
\begin{equation*}\label{s2}
\mathcal{D}(\ell^2)=\{D=\{D(\omega_1, \omega_2): \omega_1\in
\Omega_1, \omega_2\in \Omega_2\}: D \ \mbox{satisfies} \
\eqref{s1}\}.
\end{equation*}

\begin{definition}\label{Asym null}
A mapping $\Phi$: $ \mathbb{R}^+ \times \Omega_1 \times \Omega_2
\times \ell^2 \to \ell^2$ is said to be \textit{asymptotically null}
in $\mathcal{D}(\ell^2)$ if for a.e. $\omega_1\in \Omega_1,
\omega_2\in \Omega_2$, any $B(\omega_1, \omega_2)\in
\mathcal{D}(\ell^2)$, and any $\varepsilon>0$, there exist
$T(\varepsilon, \omega_1, \omega_2, B(\omega_1, \omega_2))>0$ and
$I(\varepsilon, \omega_1, \omega_2, B(\omega_1, \omega_2))\in
\mathbb{N}$ such that
\begin{equation*}
\label{3.1} \sum_{|i|>I(\varepsilon, \omega_1, \omega_2, B(\omega_1,
\omega_2))}|(\Phi(t, \theta_{1, -t}\omega_1, \theta_{2, -t}\omega_2,
u(\theta_{1, -t}\omega_1, \theta_{2, -t}\omega_2)))_i|^2\leq
\varepsilon^2,
\end{equation*}
for all $t\geq T(\varepsilon, \omega_1, \omega_2, B(\omega_1,
\omega_2))$ and $u(\omega_1, \omega_2)\in B(\omega_1, \omega_2)$.
\end{definition}

\begin{theorem}\label{existence }
Suppose that
\begin{itemize}
\item [$(\mathbf{a})$] there exists a closed measurable
 (w.r.t. the $P$-completion
of $\mathcal{F}_2$) $\mathcal{D}(\ell^2)$-pullback absorbing set $K$
in $\mathcal{D}(\ell^2)$ such that for a.e. $\omega_1\in \Omega_1,
\omega_2\in \Omega_2$, any $B(\omega_1, \omega_2)\in
\mathcal{D}(\ell^2)$, there exists $T_B(\omega_1, \omega_2)>0$
yielding\\ $\Phi(t, \theta_{1, -t}\omega_1, \theta_{2, -t}\omega_2)
B(\theta_{1, -t}\omega_1, \theta_{2, -t}\omega_2)\subset K$ for all
$t\geq T_B(\omega_1, \omega_2)$;

\item [$(\mathbf{b})$] $\Phi$: $ \mathbb{R}^+ \times \Omega_1
\times \Omega_2 \times \ell^2 \to \ell^2$ is asymptotically null on
$K$, i.e., for a.e. $\omega_1\in \Omega_1, \omega_2\in \Omega_2$,
any $B(\omega_1, \omega_2)\in \mathcal{D}(\ell^2)$, and any
$\varepsilon>0$, there exist $T(\varepsilon, \omega_1, \omega_2,
K)>0$ and $I_0(\varepsilon, \omega_1, \omega_2, K)\in \mathbb{N}$
such that
\begin{eqnarray*}
\label{3.2} &&\sup_{u\in K}\sum_{|i|>I_0(\varepsilon, \omega_1,
\omega_2, K)}|(\Phi(t, \theta_{1, -t}\omega_1, \theta_{2,
-t}\omega_2, u(\theta_{1, -t}\omega_1, \theta_{2,
-t}\omega_2)))_i|^2\leq \varepsilon^2,\nonumber\\
&& ~~~~~~~~\forall t\geq T(\varepsilon, \omega_1, \omega_2,
B(\omega_1, \omega_2)).
\end{eqnarray*}
\end{itemize}
Then
\begin{itemize}
\item [(i)] $\Phi$ possesses a unique $\mathcal{D}(\ell^2)$-pullback
attractor is given  by, for each $\omega_1  \in \Omega_1$   and
$\omega_2 \in \Omega_2$,
\begin{align*}\label{attform1}
\mathcal{A} (\omega_1, \omega_2) &= \bigcap_{\tau \ge T_K(\omega_1,
\omega_2)}
 \overline{ \bigcup_{t\ge \tau} \Phi(t, \theta_{1,-t} \omega_1, \theta_{2, -t}
 \omega_2, K(\theta_{1,-t} \omega_1, \theta_{2,-t}\omega_2  ))};
\end{align*}
\item [(ii)] the Kolmogorov $\epsilon$-entropy of $\mathcal{A} (\omega_1,
\omega_2)$ satisfies
\begin{eqnarray*}
\mathbb{K}_{\varepsilon}\leq (2I_0(\varepsilon, \omega_1, \omega_2,
K)+1)\ln(\lfloor\frac{2r_0(\omega_1,
\omega_2)\sqrt{2I_0(\varepsilon, \omega_1, \omega_2,
K)+1}}{\varepsilon}\rfloor+1),
\end{eqnarray*}
where $r_0(\omega_1, \omega_2)=\sup_{u(\omega_1, \omega_2)\in
K}\|u(\omega_1, \omega_2)\|$.
\end{itemize}
\end{theorem}

\begin{proof} The proof is based on Theorem 3.1 in \cite{HSZ} and
Proposition \ref{att} under slightly modifications.

(i) For a.e. $\omega_1\in \Omega_1, \omega_2\in \Omega_2$ and
$t_n\rightarrow\infty$ as $n\rightarrow \infty$, let $p_n(\omega_1,
\omega_2)\in K(\theta_{1, -t_n}\omega_1, \theta_{2,
-t_n}\omega_2)\in \mathcal{D}(\ell^2) \ (n=1, 2, \cdots)$ and
$$u^{(n)}(\omega_1, \omega_2)=\Phi(t_n, \theta_{1,
-t_n}\omega_1, \theta_{2, -t_n}\omega_2),$$ where
$u_i^{(n)}(\omega_1, \omega_2)=(\Phi(t_n, \theta_{1, -t_n}\omega_1,
\theta_{2, -t_n}\omega_2))_i, i\in \mathbb{Z}$. By $(\mathbf{a})$,
there exists $N_1(\omega_1, \omega_2, K)\in \mathbb{N}$ such that
$t_n\geq T_K(\omega_1, \omega_2)$ if $n\ge N_1(\omega_1, \omega_2,
K)$. Hence
 $$u^{(n)}(\omega_1, \omega_2)=\Phi(t_n, \theta_{1, -t_n}\omega_1,
\theta_{2, -t_n}\omega_2)p_n(\omega_1, \omega_2)\in K, \ \forall
n\ge N_1(\omega_1, \omega_2, K).$$ Now let us prove that the set
$$\Lambda=\{u^{(n)}(\omega_1, \omega_2)=\Phi(t_n, \theta_{1,
-t_n}\omega_1, \theta_{2, -t_n}\omega_2)p_n(\omega_1,
\omega_2)\}_{n\ge N_1(\omega_1, \omega_2, K)}$$ is pre-compact, that
is, for any given $\varepsilon>0$, $\Lambda$ has a finite covering
of balls of radius $\varepsilon$. By condition $(\mathbf{b})$, there
exists $T_1(\varepsilon, \omega_1, \omega_2, K)>0$ and
$I_0(\varepsilon, \omega_1, \omega_2, K)\in \mathbb{N}$ such that
for $n\ge N_2(\varepsilon, \omega_1, \omega_2, K)$, we have that
$t_n\ge T_1(\varepsilon, \omega_1, \omega_2, K)$ and
\begin{eqnarray*}
\label{15} \sup_{n\ge N_1(\omega_1, \omega_2,
K)}(\sum_{|i|>I_0(\varepsilon, \omega_1, \omega_2, K)}|\Phi(t_n,
\theta_{1, -t_n}\omega_1, \theta_{2, -t_n}\omega_2) p_n(\omega_1,
\omega_2))_i|^2)^{\frac{1}{2}}\leq \frac{\varepsilon}{2}.
\end{eqnarray*}
Let $N_3(\varepsilon, \omega_1, \omega_2, K)=\max\{N_1(\omega_1,
\omega_2, K), N_2(\varepsilon, \omega_1, \omega_2, K)\}$. Thus for
any $n\ge N_3(\varepsilon, \omega_1, \omega_2, K)$,
$u^{(n)}(\omega_1, \omega_2)=(u_i^{(n)}(\omega_1, \omega_2))_{i\in
\mathbb{Z}}$ can be decomposed into
\begin{eqnarray}
&&\label{16} u^{(n)}(\omega_1, \omega_2)=(u_i^{(n)}(\omega_1,
\omega_2))_{i\in \mathbb{Z}}\nonumber\\
&&~~~~=(v_i^{(n)}(\omega_1, \omega_2))_{i\in
\mathbb{Z}}+(o_i^{(n)}(\omega_1, \omega_2))_{i\in
\mathbb{Z}}\nonumber\\
&&~~~~~~=v^{(n)}(\omega_1, \omega_2)+o^{(n)}(\omega_1, \omega_2),
\end{eqnarray}
where
$$v_i^{(n)}(\omega_1, \omega_2)=\left\{
       \begin{array}{lll}
u_i^{(n)}(\omega_1, \omega_2), &|i|\le I_0(\varepsilon, \omega_1,
\omega_2, K),\\
0, &|i|> I_0(\varepsilon, \omega_1, \omega_2, K),
       \end{array}
      \right. $$
and $$o_i^{(n)}(\omega_1, \omega_2)=\left\{
       \begin{array}{lll}
0, &|i|\le I_0(\varepsilon, \omega_1,
\omega_2, K),\\
u_i^{(n)}(\omega_1, \omega_2), &|i|> I_0(\varepsilon, \omega_1,
\omega_2, K).
       \end{array}
      \right. $$
Then for $n\ge N_3(\varepsilon, \omega_1, \omega_2, K)$, we obtain
\begin{eqnarray*}
\|v^{(n)}(\omega_1, \omega_2)\|^2&=&\sum_{|i|\le I_0(\varepsilon,
\omega_1, \omega_2, K)}|u_i^{(n)}(\omega_1,
\omega_2)|^2\nonumber\\
&\le& \|u^{(n)}(\omega_1, \omega_2)\|^2\le r^2_0(\omega_1,
\omega_2),
\end{eqnarray*}
\begin{eqnarray*}
\|o^{(n)}(\omega_1, \omega_2)\|^2=\sum_{|i|> I_0(\varepsilon,
\omega_1, \omega_2, K)}|u_i^{(n)}(\omega_1, \omega_2)|^2\le
\frac{\varepsilon^2}{4},
\end{eqnarray*}
and $$|v_i^{(n)}(\omega_1, \omega_2)|\le r_0(\omega_1, \omega_2),$$
for all $|i|\le I_0(\varepsilon, \omega_1, \omega_2, K)$, where
$r_0(\omega_1, \omega_2)$ defined in (ii). Now let
\begin{eqnarray*}
\label{17} &&\Gamma(\omega_1, \omega_2)=\{v=(v_i)_{|i|\le
I_0(\varepsilon, \omega_1, \omega_2, K)}\in
\mathbb{R}^{2I_0(\varepsilon, \omega_1, \omega_2, K)+1}: \nonumber\\
&&~~~~~~v_i\in \mathbb{R}, |v_i|\le r_0(\omega_1, \omega_2)\},
\end{eqnarray*}
and
\begin{eqnarray*}
\label{18} &&n_{\varepsilon, \omega_1, \omega_2}(\Gamma(\omega_1,
\omega_2))\nonumber\\
&&~~~~~~=(\lfloor\frac{2r_0(\omega_1,
\omega_2)\sqrt{2I_0(\varepsilon, \omega_1, \omega_2,
K)+1}}{\varepsilon}\rfloor+1)^{2I_0(\varepsilon, \omega_1, \omega_2,
K)+1}.
\end{eqnarray*}
Then $\Gamma(\omega_1, \omega_2)\subset
\mathbb{R}^{2I_0(\varepsilon, \omega_1, \omega_2, K)+1}$ is a
$(2I_0(\varepsilon, \omega_1, \omega_2, K)+1)$-dimensional regular
polyhedron which is covered by $n_{\varepsilon, \omega_1,
\omega_2}(\Gamma(\omega_1, \omega_2))$ open balls of radius
$\frac{\varepsilon}{2}$ centered at $u^*_m=(u^*_{m, i})_{|i|\le
I_0(\varepsilon, \omega_1, \omega_2, K)}$, $u^*_{m, i}\in
\mathbb{R}, 1\le m\le n_{\varepsilon, \omega_1,
\omega_2}(\Gamma(\omega_1, \omega_2))$, in the norm of
$\mathbb{R}^{2I_0(\varepsilon, \omega_1, \omega_2, K)+1}$.

For each $1\le m\le n_{\varepsilon, \omega_1,
\omega_2}(\Gamma(\omega_1, \omega_2))$, we set $v_i=(v_{m,i})_{i\in
\mathbb{Z}}\in \ell^2$ such that
$$v_{m, i}=\left\{
       \begin{array}{lll}
u^*_{m, i}, &|i|\le I_0(\varepsilon, \omega_1,
\omega_2, K),\\
0, &|i|> I_0(\varepsilon, \omega_1, \omega_2, K).
       \end{array}
      \right. $$
Then for $v^{(n)}(\omega_1, \omega_2)=(v_i^{(n)}(\omega_1,
\omega_2))_{i\in \mathbb{Z}}\ (n\ge N_3(\varepsilon, \omega_1,
\omega_2, K))$ in the decomposition \eqref{16}, there exists $m_0\in
\{1, 2, \cdots, n_{\varepsilon, \omega_1, \omega_2}(\Gamma(\omega_1,
\omega_2))\}$ such that
\begin{eqnarray*}
\|v^{(n)}(\omega_1, \omega_2)-v_{m_0}\|^2=\sum_{|i|\le
I_0(\varepsilon, \omega_1, \omega_2, K)}|u_i^{(n)}(\omega_1,
\omega_2)-u_{m_0, i}|^2\le \frac{\varepsilon^2}{4},
\end{eqnarray*}
and hence, we get
\begin{eqnarray*}
\label{19} &&\|u^{(n)}(\omega_1,
\omega_2)-v_{m_0}\|^2=\|v^{(n)}(\omega_1,
\omega_2)-v_{m_0}+o^{(n)}(\omega_1, \omega_2)\|^2\nonumber\\
&&~~~~~~~~~\leq 2\|v^{(n)}(\omega_1,
\omega_2)-v_{m_0}\|^2+2\|o^{(n)}(\omega_1, \omega_2)\|^2\le
\varepsilon^2.
\end{eqnarray*}
Therefore, $\{u_i^{(n)}(\omega_1, \omega_2)=\Phi(t_n, \theta_{1,
-t_n}\omega_1, \theta_{2, -t_n}\omega_2)_{n\ge N_3(\varepsilon,
\omega_1, \omega_2, K)}\}\subset \ell^2$ can be covered by\\
$n_{\varepsilon, \omega_1, \omega_2}(\Gamma(\omega_1, \omega_2))$
open balls of radius $\varepsilon$ centered at $v_m=(v_{m, i})_{i\in
\mathbb{Z}}, 1\le m\le n_{\varepsilon, \omega_1,
\omega_2}(\Gamma(\omega_1, \omega_2))$.

(ii) By the invariant property of \textit{$\mathcal{D}(X)$-pullback
attractors}, we have $$\mathcal{A}(\omega_1, \omega_2)=\Phi(t,
\theta_{1, -t}\omega_1, \theta_{2,
-t}\omega_2)\mathcal{A}(\theta_{1, -t}\omega_1, \theta_{2,
-t}\omega_2)\subset K$$ for $t\ge T(\omega_1, \omega_2, K)$ and a.e.
$\omega_1\in \Omega_1, \omega_2\in \Omega_2$. For any
$\varepsilon>0$, we can see that $\mathcal{A}(\omega_1, \omega_2)$
can be covered under the norm of $\ell^2$, by $n_{\varepsilon,
\omega_1, \omega_2}(\Gamma(\omega_1, \omega_2))$-balls in $\ell^2$
with center $v_m=(v_{m, i})_{i\in \mathbb{Z}}, 1\le m\le
n_{\varepsilon, \omega_1, \omega_2}(\Gamma(\omega_1, \omega_2))$ and
radius $\varepsilon$. Thus, by the definition of \eqref{Kol}, the
proof is completed.
\end{proof}

\section{Pullback attractors for lattice differential equations in
$\ell^2$}\label{lattice equations} In this section, we discuss the
proper choice of parametric spaces $\Omega_1$ and $\Omega_2$ to
consider pullback attractors for lattice differential equations with
both non-autonomous deterministic and random forcing terms by using
the abstract theory presented in the previous section.

Suppose now $\Omega_1=\mathbb{R}$. Define a family $\{\theta_{1,
t}\}_{t\in \mathbb{R}}$ of shift operators by
\begin{equation}\label{20}
 \theta_{1, t}(\tau)=\tau+t, \ \ \forall t, \tau\in \mathbb{R}.
\end{equation}
Let $\Phi$: $ \mathbb{R}^+ \times \mathbb{R} \times \Omega_2 \times
\ell^2 \to \ell^2$ be a continuous cocycle on $\ell^2$ over
$(\mathbb{R}, \{\theta_{1, t}\}_{t\in \mathbb{R}})$ and $(\Omega_2,
\mathcal{F}_2, P, \{\theta_{2, t}\}_{t\in \mathbb{R}})$ where
$\{\theta_{1, t}\}_{t\in \mathbb{R}}$ is defined in \eqref{20}. Due
to Theorem \ref{existence }, we obtain the following result:

\begin{theorem}\label{existence 2}
Suppose that
\begin{itemize}
\item [$(\mathbf{a})$] there exists a closed measurable
 (w.r.t. the $P$-completion
of $\mathcal{F}_2$) $\mathcal{D}(\ell^2)$-pullback absorbing set $K$
in $\mathcal{D}(\ell^2)$ such that for a.e. $\tau\in \mathbb{R},
\omega\in \Omega_2$, any $B(\tau, \omega)\in \mathcal{D}(\ell^2)$,
there exists $T_B(\tau, \omega)>0$ yielding $\Phi(t, \tau-t,
\theta_{2, -t}\omega) B(\tau-t, \theta_{2, -t}\omega)\subset K$ for
all $t\geq T_B(\tau, \omega)$;

\item [$(\mathbf{b})$] $\Phi$: $ \mathbb{R}^+ \times \mathbb{R}
\times \Omega_2 \times \ell^2 \to \ell^2$ is asymptotically null on
$K$, i.e., for a.e. $\tau\in \mathbb{R}, \omega\in \Omega_2$, any
$B(\tau, \omega)\in \mathcal{D}(\ell^2)$, and any $\varepsilon>0$,
there exist $T(\varepsilon, \tau, \omega, K)>0$ and
$I_0(\varepsilon, \tau, \omega, K)\in \mathbb{N}$ such that
\begin{eqnarray}
\label{24} &&\sup_{u\in K}\sum_{|i|>I_0(\varepsilon, \tau, \omega,
K)}|(\Phi(t, \tau-t, \theta_{2, -t}\omega, u(\tau-t, \theta_{2,
-t}\omega)))_i|^2\leq \varepsilon^2,\nonumber\\
&& ~~~~~~~~\forall t\geq T(\varepsilon, \tau, \omega, B(\tau,
\omega)).
\end{eqnarray}
\end{itemize}
Then
\begin{itemize}
\item [(i)] $\Phi$ possesses a unique $\mathcal{D}(\ell^2)$-pullback
attractor is given  by, for each $\tau  \in \mathbb{R}$   and
$\omega \in \Omega_2$,
\begin{align}\label{25}
\mathcal{A} (\tau, \omega) &= \bigcap_{s \ge T_K(\tau, \omega)}
 \overline{ \bigcup_{t\ge s} \Phi(t, \tau-t, \theta_{2, -t}
 \omega, K(\tau-t, \theta_{2,-t}\omega  ))};
\end{align}
\item [(ii)] the Kolmogorov $\varepsilon$-entropy of $\mathcal{A} (\tau,
\omega)$ satisyies
\begin{eqnarray*}
\mathbb{K}_{\varepsilon}\leq (2I_0(\varepsilon, \tau, \omega,
K)+1)\ln(\lfloor\frac{2r_0(\tau, \omega)\sqrt{2I_0(\varepsilon,
\tau, \omega, K)+1}}{\varepsilon}\rfloor+1),
\end{eqnarray*}
where $r_0(\tau, \omega)=\sup_{u(\tau, \omega)\in K}\|u(\tau,
\omega)\|, \ \forall \tau\in \mathbb{R}, \omega\in \Omega_2$.
\end{itemize}
\end{theorem}

\section{Pullback attractors for SLDS in $\ell^2$}\label{SLDS}
In this section, we will apply Theorem \ref{existence 2} to prove
the existence of a pullback attractor for non-autonomous first order
stochastic lattice dynamical system.

\subsection{Mathematical Settings}

Denote $C_b(\mathbb{R}, \ell^2)$ be the space of all continuous
bounded functions from $\mathbb{R}$ into $\ell^2$. Consider the
following non-autonomous first order lattice differential equations
with time-dependent external forcing terms and multiplicative white
noise
\begin{equation}
\label{5.2} \frac{du_i(t)}{dt}=\nu_i(t)
(u_{i-1}-2u_i+u_{i+1})-\lambda_i(t) u_i-f_i(u_i,
t)+g_i(t)+u_i\circ\frac{dw(t)}{dt}, \ \ i\in \mathbb{Z},
\end{equation}
with initial data
\begin{equation}
\label{5.3} u_i(\tau)=u_{i, \tau}, \ \ i\in \mathbb{Z}, \tau\in
\mathbb{R},
\end{equation}
where $u_i\in \mathbb{R}$, $\mathbb{Z}$ denotes the integer set;
$\nu_i(t)$ and $\lambda_i(t)$ are locally integrable in $t$; $g_i\in
C(\mathbb{R}, \mathbb{R})$ and $f_i\in C(\mathbb{R}\times
\mathbb{R}, \mathbb{R})$ for $i\in \mathbb{Z}$; $w$ is an
independent Brownian motion. Note that system
\eqref{5.2}-\eqref{5.3} can be written as for $t\ge \tau\in
\mathbb{R}$,
\begin{equation}
\label{5.4} \frac{du}{dt}=-\nu(t) Au-\lambda(t) u-f(u,
t)+g(t)+u\circ\frac{dw(t)}{dt}, \ u(\tau)=u_{\tau}=(u_{i,
\tau})_{i\in \mathbb{Z}},
\end{equation}
where $u=(u_{i})_{i\in \mathbb{Z}}$, $f(u, t)=(f_i(u_i, t))_{i\in
\mathbb{Z}}, g(t)=(g_i(t))_{i\in \mathbb{Z}}$,
$Au=(-u_{i-1}+2u_i-u_{i+1})_{i\in \mathbb{Z}}$ and $w(t)$ is the
white noise with values in $\ell^2$ defined on the probability space
$(\Omega, \mathcal{F}, P)$ and
$$\Omega=\{\omega\in C(\mathbb{R}, \ell^2): \omega(0)=0\},$$ the
Borel sigma-algebra $\mathcal{F}$ is generated by the compact open
topology, and $P$ is the corresponding Wiener measure on
$\mathcal{F}$. Define a group $\{\theta_{2, t}\}_{t\in \mathbb{R}}$
acting on $(\Omega, \mathcal{F}, P)$ by
\begin{equation}\label{5.6}
\theta_{2, t}\omega(\cdot)=\omega(\cdot+t)-\omega(t), \ \ \omega\in
\Omega, t\in \mathbb{R}.
\end{equation}
Then $(\Omega, \mathcal{F}, P, \{\theta_{2, t}\}_{t\in \mathbb{R}})$
is a parametric dynamical system. We make the following assumptions:

\begin{itemize}
\item [$(\mathbf{A1})$] $\lambda_i(t), \nu_i(t)\in L_{loc}^1(\mathbb{R})$ in $t$
and there exist positive constants $\lambda^0, \lambda_0$ and
$\nu^0, \nu_0$ such that for $\forall i\in \mathbb{Z}$, $t\in
\mathbb{R}$,
$$0<\lambda_0\le \lambda_i(t)\le \lambda^0<+\infty, $$
$$0<\nu_0\le \nu_i(t)\le \nu^0<+\infty;$$
\item [$(\mathbf{A2})$] $f_i(x, t)$ is differentiable in $x$ and
continuous in $t$; $f_i(0, t)=0$; $xf_i(x, t)\ge -\alpha_i^2(t)$,
where $\alpha(t)=(\alpha_i(t))_{i\in \mathbb{Z}}\in C_b(\mathbb{R},
\ell^2)$, and there exists a constant $\beta\ge 0$ such that
$\partial_x f_i(x, t)\ge -\beta$, $\forall x, t\in \mathbb{R}, i\in
\mathbb{Z}$;
\item [$(\mathbf{A3})$] There exists a positive-valued continuous
function $\zeta(\iota, t)\in C(\mathbb{R}^+\times \mathbb{R},
\mathbb{R}^+)$ such that
\begin{equation*}
\sup_{i\in \mathbb{Z}}\max_{x\in [-\iota, \iota]}|\partial_x f_i(x,
t)|\le \zeta(\iota, t), \ \forall \iota\in \mathbb{R}^+, t\in
\mathbb{R};
\end{equation*}
\item [$(\mathbf{A4})$] $g(t)=(g_i(t))_{i\in \mathbb{Z}}\in
C_b(\mathbb{R}, \ell^2)$.
\end{itemize}

Now, let $\{\theta_{1, t}\}_{t\in \mathbb{R}}$ be the group acting
on $\mathbb{R}$ given by \eqref{20}. We next define a continuous
cocycle for system \eqref{5.4} over $(\mathbb{R}, \{\theta_{1,
t}\}_{t\in \mathbb{R}})$ and $(\Omega_2, \mathcal{F}_2, P,
\{\theta_{2, t}\}_{t\in \mathbb{R}})$. This can be done by first
transferring the stochastic system into a corresponding
non-autonomous deterministic one. Given $\omega\in \Omega$, denote
by
\begin{equation}
\label{5.7} z(\omega)=-\int^0_{-\infty}e^{r}\omega(r)dr.
\end{equation}
Then the random variable $z$ given in \eqref{5.7} is a stationary
solution of the one-dimensional Ornstein-Uhlenbeck equation
\begin{equation*}
dz+zdt=dw(t).
\end{equation*}
In other words, we get
\begin{equation}\label{5.8}
dz(\theta_{2, t}\omega)+z(\theta_{2, t}\omega)dt=dw(t).
\end{equation}
By \cite{BLL, CL}, we know that there exists a $\theta_{2,
t}$-variant set $\Omega'\subseteq \Omega$ of full $P$ measure such
that $z(\theta_{2, t}\omega)$ is continuous in $t$ for every
$\omega\in \Omega'$, and the random variable $|z(\omega)|$ is
tempered. In addition, for every $\omega\in \Omega'$, we have the
following limits:
\begin{equation}\label{temperness}
\lim_{t\rightarrow\pm\infty}\frac{|\omega(t)|}{|t|}=0, \quad
\lim_{t\rightarrow\pm\infty}\frac{|z(\theta_{2, t}\omega)|}{|t|}=0
\quad \mbox{and} \quad
\lim_{t\rightarrow\pm\infty}\frac{1}{t}\int_0^tz(\theta_{2,
s}\omega)ds=0.
\end{equation}
Hereafter, we will write $\Omega$ as $\Omega'$ and $\theta_{2, t}$
as $\theta_t$ instead.

\subsection{Existence and Uniqueness of a Mild Solution}

Let $u(t)$ be the solution of system \eqref{5.4}, then
$v(t)=u(t)e^{-z(\theta_t\omega)}$ satisfies
\begin{equation}
\label{5.9} \frac{dv}{dt}=-\nu(t) Av-\lambda(t)
v-e^{-z(\theta_t\omega)}f(e^{z(\theta_t\omega)}v,
t)+e^{-z(\theta_t\omega)}g(t)+z(\theta_t\omega)v,
\end{equation}
with initial condition $v_{\tau}=v(\tau,
\omega)=u_{\tau}e^{-z(\theta_{\tau}\omega)}, \ t>\tau, \tau\in
\mathbb{R}, \omega\in \Omega$. We recall $v: [\tau,
\tau+T)\rightarrow \ell^2\ (T>0)$ a mild solution of the following
random differential equation
$$\frac{dv(t)}{dt}=G(v, t, \theta_t\omega), \ \ v=(v_i)_{i\in
\mathbb{Z}}, G=(G_i)_{i\in \mathbb{Z}}, \ \ t\ge \tau\in
\mathbb{R},$$ where $\omega\in \Omega$, if $v\in C([\tau, \tau+T),
\ell^2)$ and
$$v_i(t, \tau)=v_i(\tau)+\int_{\tau}^tG_i(v(s), s, \theta_s\omega)ds
 \ \ \text{for}\ i\in \mathbb{Z}\ \text{and} \ t\in [\tau, \tau+T).$$

In this subsection, we will prove the existence and uniqueness of
the mild solution of system \eqref{5.9}.

\begin{proposition}\label{uniqueness}
Let $T>0$ and assumptions $(\mathbf{A1}$-$\mathbf{A4})$ hold. Then
for $\tau\in \mathbb{R}, \omega\in \Omega$ and any initial data
$v_{\tau}\in \ell^2$, system \eqref{5.9} has a unique $(\mathcal{F},
\mathcal{B}(\ell^2))$-measurable mild solution $v(\cdot, \tau;
\omega, v_{\tau}$, $g)\in C([\tau, \tau+T), \ell^2)$ with $v(\tau,
\tau; \omega, v_{\tau}, g)=v_{\tau}$, $v(t, \tau; \omega, v_{\tau},
g)\in \ell^2$ being continuous in $v_{\tau}\in \ell^2$ and $g\in
C_b(\mathbb{R}, \ell^2)$. Moreover, the solution $v(t, \tau; \omega,
v_{\tau}, g)$ exists globally on $[\tau, +\infty)$ for any $\tau\in
\mathbb{R}$. Moreover, for given $t\in \mathbb{R}^+, \tau\in
\mathbb{R}, \omega\in \Omega$ and $u_{\tau}\in \ell^2$, the mapping
$$\Phi(t, \tau, \omega, v_{\tau}, g)=v(t+\tau, \tau;
\theta_{-\tau}\omega, v_{\tau}, g) =u(t+\tau, \tau;
\theta_{-\tau}\omega, u_{\tau}, g)e^{-z(\theta_t\omega)},$$
generates a continuous cocycle from $\mathbb{R}^+\times
\mathbb{R}\times \Omega\times \ell^2$ to $\ell^2$ over $(\mathbb{R},
\{\theta_{1, t}\}_{t\in \mathbb{R}})$ and $(\Omega, \mathcal{F}, P,
\{\theta_{t}\}_{t\in \mathbb{R}})$, where
$v_{\tau}=u_{\tau}e^{-z(\theta_{\tau}\omega)}$.
\end{proposition}

\begin{proof}
We first show that if $v_{\tau}\in \ell^2$, system \eqref{5.9} has a
unique measurable mild solution $v(t, \tau; \omega, v_{\tau}, g)\in
\ell^2$ on $[\tau, \tau+T)$ with $v(\tau, \tau; \omega, v_{\tau},
g)=v_{\tau}$ for $T>0$ and $\omega\in \Omega$.

Given $\omega\in \Omega, v_{\tau}\in \ell^2$ and $g\in
C_b(\mathbb{R}, \ell^2)$, let
$$F(v, t, \omega)=-\nu(t) Av-\lambda(t) v-e^{-z(\omega)}f(ve^{z(\omega)},
t)+e^{-z(\omega)}g(t)+vz(\omega).$$ Note that $F(v, t, \omega)$ is
continuous in $v$ and locally integrable in $t$ and measurable in
$\omega$ from $\ell^2\times \mathbb{R}\times\Omega$ into $\ell^2$.
Denote $|\| \cdot|\|=\sup_{t\in \mathbb{R}}\|\cdot(t)\|$, then by
$(\mathbf{A1}$-$\mathbf{A4})$,
\begin{eqnarray*}
\begin{split} \|F(v, t, \omega)\|&\le
(\lambda^0+4\nu^0+\max\{\zeta(\|v\||e^{z(\omega)}|, t),
\beta\}+|z(\omega)|) \|v\|\\
&\quad+|e^{-z(\omega)}||\| g|\|.
\end{split}
\end{eqnarray*} Hence for any
$v^{(1)}=(v_i^{(1)})_{i\in \mathbb{Z}}, v^{(2)}=(v_i^{(2)})_{i\in
\mathbb{Z}}\in \ell^2$,
\begin{eqnarray*}
&&\|F(v^{(1)}, t, \omega)-F(v^{(2)}, t, \omega)\|\\
&&~~\le
(\lambda^0+4\nu^0+\max\{\zeta((\|v^{(1)}\|+\|v^{(2)}\|)|e^{z(\omega)}|,
t), \beta\}+|z(\omega)|)\|v^{(1)}-v^{(2)}\|.
\end{eqnarray*}
For any bounded set $B\subset \ell^2$ with $\sup_{u\in B}\|v\|\le
\iota$, and define
$$\kappa_B(t, \omega)=(\lambda^0+4\nu^0+\max\{\zeta(\iota|e^{z(\omega)}|,
t), \beta\}+|z(\omega)|)\iota+|e^{-z(\omega)}||\| g|\|\ge 0,$$ then
for any $v, v^{(1)}, v^{(2)}\in B$, $$F(v, t, \omega)\leq
\kappa_B(t, \omega), \ \ \|F(v^{(1)}, t, \omega)-F(v^{(2)}, t,
\omega)\|\le \kappa_B(t, \omega)\|v^{(1)}-v^{(2)}\|$$ and
$$\int_{\tau}^{\tau+1}\kappa_B(s, \theta_s\omega)ds<\infty, \
\forall \tau\in \mathbb{R}.$$ By \cite[Proposition 2.1.1]{Chueshov},
problem \eqref{5.9} possesses a unique local mild solution $v(\cdot,
\tau, \omega; v_{\tau}, g)\in C([\tau, \tau+T_{\max}), \ell^2)$
satisfying the integral equation
\begin{eqnarray}
\begin{split}
\label{5.10} v(t)&=v_{\tau}+\int_{\tau}^t(-\nu(s) Av-\lambda(s)
v-e^{-z(\theta_s\omega)}f(ve^{z(\theta_s\omega)}, s)\\
&\quad\quad +e^{-z(\theta_s\omega)}g(s) +v z(\theta_s\omega))ds, \
t\in [\tau, \tau+T_{\max}) \ (0<T_{\max}\le T),
\end{split}
\end{eqnarray}
where $[\tau, \tau+T_{\max})$ is the maximal interval of existence
of the solution of \eqref{5.9}.

We next show that $T_{\max}=T$. Since $\lambda_i(t), \nu_i(t)\in
L_{loc}^1(\mathbb{R})$ in $t$, by \cite{Pazy}, there exist sequences
of continuous functions in $t\in \mathbb{R}$, $\lambda_i^{(m)}(t),
\nu_i^{(m)}(t), m\in \mathbb{N}$, such that
\begin{equation}\label{cov1}
\lim_{m\rightarrow\infty}\int_{\tau}^t|\lambda_i^{(m)}(s)-\lambda_i(s)|ds=0
\ \mbox{and} \ \lambda_0\le \lambda_i^{(m)}(t)\le \lambda^0, \forall
\tau, t\in \mathbb{R},
\end{equation}
\begin{equation}\label{cov2}
\lim_{m\rightarrow\infty}\int_{\tau}^t|\nu_i^{(m)}(s)-\nu_i(s)|ds=0
\ \mbox{and} \ \nu_0\le \nu_i^{(m)}(t)\le \nu^0, \forall \tau, t\in
\mathbb{R}.
\end{equation}

Consider the following differential equations with initial data
$v_{\tau}\in \ell^2$,
\begin{equation}\label{int1}
\frac{dv^{(m)}}{dt}=F^{(m)}(v^{(m)}, t, \omega),
\end{equation}
where $F^{(m)}(v^{(m)}, t, \omega)=(F^{(m)}_i(v^{(m)}, t,
\omega))_{i\in \mathbb{Z}}$ and
\begin{equation}\label{int2}
\begin{split}
F^{(m)}_i(v^{(m)}, t, \omega)&=-\nu_i^{(m)}(t)
Av_i^{(m)}-\lambda_i^{(m)}(t)
v_i^{(m)}\\
&\quad-e^{-z(\omega)}f_i(v_i^{(m)}e^{z(\omega)},
t)+e^{-z(\omega)}g_i(t)+v_i^{(m)}z(\omega).
\end{split}
\end{equation}
For $\omega\in \Omega$, by the continuity of $F^{(m)}_i(v^{(m)}, t,
\omega)$ in $t$, \eqref{int1} has a unique solution $v(\cdot, \tau;
\omega, v_{\tau}, g)\in C([\tau, \tau+T_{\max}^{(m)}), \ell^2)\cap
C^1((\tau, \tau+T_{\max}^{(m)}), \ell^2)$ such that
\begin{equation}\label{int3}
\begin{split}
\frac{dv_i^{(m)}}{dt}=F^{(m)}_i(v^{(m)}, t, \omega)
\end{split}
\end{equation}
and
\begin{equation}\label{int4}
\begin{split}
v_i^{(m)}=v_{\tau}+\int_{\tau}^tF^{(m)}_i(v^{(m)}(s), s, \omega)ds.
\end{split}
\end{equation}
Taking the inner product in $\ell^2$ in \eqref{int3} yields
\begin{eqnarray}\label{5.11}
&&\frac{d\|v^{(m)}\|^2}{dt}= 2(-\nu^{(m)}(t)
Av^{(m)}-\lambda^{(m)}(t) v^{(m)}+z(\theta_t\omega)v^{(m)}, v^{(m)})\nonumber\\
&&\quad-2(e^{-z(\theta_t\omega)}f^{(m)}(v^{(m)}e^{z(\theta_t\omega)},
t), v^{(m)})+2(e^{-z(\theta_t\omega)}g^{(m)}(t), v^{(m)}).
\end{eqnarray}
Note that
\begin{equation*}
\begin{split}
-\beta e^{2z(\theta_t\omega)}\|v^{(m)}\|^2&\le
(f(v^{(m)}e^{z(\theta_t\omega)}, t),
v^{(m)}e^{z(\theta_t\omega)})\\
&\quad\quad\le \iota(e^{z(\theta_t\omega)}\|v^{(m)}\|,
s)e^{2z(\theta_t\omega)}\|v^{(m)}\|^2.
\end{split}
\end{equation*}
It follows from \eqref{5.11} that
\begin{eqnarray}\label{5.12}
\frac{d\|v^{(m)}\|^2}{dt}\le
(-\lambda_0+2\beta+2z(\theta_s\omega))\|v^{(m)}\|^2+(2|\|
\alpha|\|^2+\frac{|\| g|\|^2}{\lambda_0})e^{-2z(\theta_s\omega)}.
\end{eqnarray}
Applying Gronwall's inequality to \eqref{5.12}, we obtain that
\begin{eqnarray*}
\begin{split}
&\|v^{(m)}(t)\|^2\le \|v_{\tau}\|^2e^{(2\beta-\lambda_0)(t-\tau)
+2\int_{\tau}^tz(\theta_r\omega)dr}\\
&\quad+(2|\| \alpha|\|^2+\frac{|\|
g|\|^2}{\lambda_0})e^{(2\beta-\lambda_0)t
+2\int_{0}^tz(\theta_r\omega)dr}\int_{\tau}^te^{(\lambda_0-2\beta)s
-2z(\theta_s\omega) -2\int_{0}^sz(\theta_r\omega)dr}ds\\
&\quad\quad:=\eta^2(t, \tau, \omega), \quad t\in [\tau,
\tau+T_{\max}^{(m)}),
\end{split}
\end{eqnarray*}
where $\eta^2(t, \tau, \omega)\in C([\tau, \tau+T), \mathbb{R}^+)$
is independent of $m$, which implies that
\begin{eqnarray}\label{5.13}
\begin{split}
|v_i^{(m)}(t)|\le \eta(t, \tau, \omega), \ \mbox{for all} \ m\in
\mathbb{N}, t\in [\tau, \tau+T), \omega\in \Omega.
\end{split}
\end{eqnarray}
It then follows that for some $ \tilde{\eta}(T, \tau, \omega)>0$,
which is independent on $m$ such that $|F_i^{(m)}(v^{(m)}(t), t)|\le
\tilde{\eta}(T, \tau, \omega)$ and
\begin{eqnarray*}
\begin{split}
|v_i^{(m)}(t)-v_i^{(m)}(s)|&=\int_s^t|F_i^{(m)}(v^{(m)}(r), r)|dr\\
&\le \tilde{\eta}(T, \tau, \omega)|t-s|, \forall t, s\in [\tau,
\tau+T), m\in \mathbb{N}, \omega\in \Omega.
\end{split}
\end{eqnarray*}
By the Arzela-Ascoli Theorem, there exists a convergent subsequence
$\{v_i^{(m_k)}(t)$, $t\in [\tau, \tau+T)\}$ of $\{v_i^{(m)}(t), t\in
[\tau, \tau+T)\}$ such that
\begin{eqnarray*}
\begin{split}
v_i^{(m_k)}(t)\rightarrow \bar{v}_i(t) \ \mbox{as} \
k\rightarrow\infty \ \mbox{for} \ t\in [\tau, \tau+T), i\in
\mathbb{Z}
\end{split}
\end{eqnarray*}
and $\bar{v}_i(t)$ is continuous in $ t\in [\tau, \tau+T)$.
Moreover, $|\bar{v}_i(t)|\le  \eta(t, \tau, \omega)$ for $t\in
[\tau, \tau+T), \omega\in \Omega$.

By \eqref{cov1}, \eqref{cov2}, \eqref{5.13} and the Lebesgue
Dominated Convergence Theorem, we have
\begin{equation}\label{cov3}
\lim_{k\rightarrow\infty}\int_{\tau}^t|\lambda_i^{(m_k)}(s)v_i^{(m_k)}(s)
-\lambda_i(s)\bar{v}_i(s)|ds=0,
\end{equation}
\begin{equation}\label{cov4}
\lim_{k\rightarrow\infty}\int_{\tau}^t|\nu_i^{(m_k)}(s)v_i^{(m_k)}(s)
-\nu_i(s)\bar{v}_i(s)|ds=0.
\end{equation}
Thus by replacing $m$ by $m_k$ in \eqref{int4} and letting
$k\rightarrow\infty$, we obtain
\begin{equation*}
\begin{split}
\bar{v}_i(t)=v_{\tau}+\int_{\tau}^tF_i(\bar{v}(s), s, \omega)ds \
\mbox{for all} \ t\in [\tau, \tau+T), \omega\in \Omega,
\end{split}
\end{equation*}
which implies that $\bar{u}(t)=(\bar{u}_i(t))_{i\in \mathbb{Z}}$ is
a mild solution of \eqref{5.9}. Then by the uniqueness of the mild
solutions of \eqref{5.9}, $T_{\max}=T$. Moreover, this means that
$v(t, \tau; \omega, v_{\tau}, g)$ exists globally on $[\tau,
+\infty)$ for any $\tau\in \mathbb{R}$. Here we remain to show for
given $t\in \mathbb{R}^+, \tau\in \mathbb{R}, \omega\in \Omega$ and
$u_{\tau}\in \ell^2$, the mapping
$$\Phi(t, \tau, \omega, v_{\tau}, g)=v(t+\tau, \tau;
\theta_{-\tau}\omega, v_{\tau}, g) =u(t+\tau, \tau;
\theta_{-\tau}\omega, u_{\tau}, g)e^{-z(\theta_t\omega)},$$
generates a continuous cocycle from $\mathbb{R}^+\times
\mathbb{R}\times \Omega\times \ell^2$ to $\ell^2$ over $(\mathbb{R},
\{\theta_{1, t}\}_{t\in \mathbb{R}})$ and $(\Omega, \mathcal{F}, P,
\{\theta_{t}\}_{t\in \mathbb{R}})$ in the sense of Definition
\ref{Def1}. In fact, the function $F(v, t, \omega)$ is continuous in
$v, g$ and measurable in $t, \omega$, which implies that $v:
(\mathbb{R}^+)\times\mathbb{R}\times\Omega\times\ell^2\rightarrow\ell^2$,
$(t, \cdot; \omega, v_{\tau}, g)\mapsto v(t, \cdot; \omega,
v_{\tau}, g)$ is
$(\mathcal{B}(\mathbb{R}^+)\times\mathcal{F}\times\mathcal{B}(\ell^2),
\mathcal{B}(\ell^2))$-measurable (see \cite{Arnold}). The proof is
complete.

\end{proof}

\subsection{Existence of a Pullback Absorbing Set}
In this subsection, we will get the existence of a
$\mathcal{D}(\ell^2)$-pullback absorbing set for the continuous
cocycle $\Phi$.

\begin{lemma}\label{abs set} Let
$\tilde{\lambda}=\lambda_0-\beta-\frac{2}{\sqrt{\pi}}>0$. Assume
that $(\mathbf{A1}$-$\mathbf{A4})$ hold, then there exists a closed
measurable $\mathcal{D}(\ell^2)$-pullback absorbing set
$\mathcal{K}=\{\mathcal{K}(\tau, \omega): \tau\in \mathbb{R},
\omega\in \Omega\}$ for $\Phi$ in $\mathcal{D}(\ell^2)$ such that
for any $B(\tau, \omega)\in \mathcal{D}(\ell^2)$, there exists
$T_B=T_B(\tau, \omega)>0$ yielding $\Phi(t, \tau-t,
\theta_{-t}\omega)B(\tau-t, \theta_{-t}\omega)\subseteq
\mathcal{K}(\tau, \omega)$ for all $t\ge T_B$ and $v_{\tau-t}\in
B(\tau-t, \theta_{-t}\omega)$.
\end{lemma}

\begin{proof}
Let $\Phi^{(m)}$ be a solution of system \eqref{int1}, then
$\Phi^{(m)}\in \ell^2$ for all $t\ge \tau$. From Proposition
\ref{uniqueness}, we know that $v^{(m)}(\tau, \tau-t,
\theta_{-\tau}\omega)=\Phi^{(m)}(t, \tau-t, \theta_{-t}\omega)$.
Denote $\hat{\lambda}=\lambda_0-\beta$ and apply Gronwall's
inequality over $(\tau-t, \tau)$ to \eqref{5.12}, it follows that
\begin{eqnarray*}
\begin{split}
&\|v^{(m)}(\tau, \tau-t, \omega,
v_{\tau-t})\|^2+\beta\int_{\tau-t}^{\tau}e^{-\hat{\lambda}(\tau-s)
+2\int_s^{\tau}z(\theta_r\omega)} \|v^{(m)}(s, \tau-t, \omega,
v_{\tau-t})\|^2ds\\
&\quad\le e^{-\hat{\lambda}t
-2\int_{\tau}^{\tau-t}z(\theta_r\omega)dr}\|v_{\tau-t}\|^2\\
&\quad\quad+(2|\| \alpha|\|^2+\frac{|\|
g|\|^2}{\lambda_0})e^{-\hat{\lambda}\tau
+2\int_{0}^{\tau}z(\theta_r\omega)dr}\int_{\tau-t}^{\tau}e^{\hat{\lambda}s
-2z(\theta_s\omega) -2\int_{0}^sz(\theta_r\omega)dr}ds.
\end{split}
\end{eqnarray*}
Since $v^{(m_k)}\rightarrow v$ for some $m_k\rightarrow\infty$,
where $v$ is the mild solution of \eqref{5.9}, then the estimation
above still holds with $v^{(m_k)}$ being replaced by $v$. Now, by
replacing $\omega$ with $\theta_{-\tau}\omega$ in the expression
$v$, we obtain
\begin{equation}\label{est2}
\begin{split}
&\|v(\tau, \tau-t, \theta_{-\tau}\omega, v_{\tau-t}\|^2\\
&\quad\quad+\beta\int_{\tau-t}^{\tau}e^{-\hat{\lambda}(\tau-s)
+2\int_s^{\tau}z(\theta_{r-\tau}\omega)dr}
\|v(s, \tau-t, \theta_{-\tau}\omega, v_{\tau-t})\|^2ds\\
&=\|v(\tau, \tau-t, \theta_{-\tau}\omega, v(\tau-t,
\theta_{-\tau}\omega)\|^2\\
&\quad\quad+\beta\int_{-t}^{0}e^{-\hat{\lambda}s
+2\int_{-t}^{0}z(\theta_{r}\omega)dr}
\|v(s+\tau, \tau-t, \theta_{-\tau}\omega, v_{\tau-t})\|^2ds\\
&\le e^{-\hat{\lambda}t
-2\int_{\tau}^{\tau-t}z(\theta_{r-\tau}\omega)dr}
\|v_{\tau-t}\|^2\\
&\quad\quad+(2|\| \alpha|\|^2+\frac{|\|
g|\|^2}{\lambda_0})\int_{\tau-t}^{\tau}e^{\hat{\lambda}(s-\tau)
-2z(\theta_{s-\tau}\omega)+2\int_{s}^{\tau}z(\theta_{r-\tau}\omega)dr}ds\\
& \le e^{-\hat{\lambda}t-2\int_{-t}^{0}z(\theta_{r}\omega)dr}
\|v_{\tau-t}\|^2\\
&\quad\quad+(2|\| \alpha|\|^2+\frac{|\|
g|\|^2}{\lambda_0})\int_{-t}^{0}e^{\hat{\lambda}s
-2z(\theta_{s}\omega)+2\int_{s}^{0}z(\theta_{r}\omega)dr}ds\\
&\le e^{-\hat{\lambda}t-2\int_{-t}^{0}z(\theta_{r}\omega)dr}
\|v_{\tau-t}\|^2\\
&\quad\quad+(2|\| \alpha|\|^2+\frac{|\|
g|\|^2}{\lambda_0})\int_{-\infty}^{0}e^{\hat{\lambda}s
-2z(\theta_{s}\omega)+2\int_{s}^{0}z(\theta_{r}\omega)dr}ds.
\end{split}
\end{equation}
Due to \eqref{temperness}, we know that
$$\int_{-\infty}^0e^{\hat{\lambda}s
-2z(\theta_{s}\omega)+2\int_{s}^0z(\theta_{r}\omega)dr}ds<+\infty,$$
and
$$\lim_{t\rightarrow\pm\infty}\frac{1}{t}\int_0^t|z(\theta_{s}\omega)|ds
=\frac{1}{\sqrt{\pi}}.$$ Let
$\tilde{\lambda}=\lambda_0-\beta-\frac{2}{\sqrt{\pi}}$ and consider
for any $v_{\tau-t}\in B(\tau-t, \theta_{-t}\omega)$, we have for
$\tilde{\lambda}>0$, $\tau\in \mathbb{R}$ from \eqref{s1} that
\begin{eqnarray}\label{tempered}
&&e^{-\hat{\lambda}t-2\int_{-t}^0z(\theta_{r}\omega)dr}
\|v_{\tau-t}\|^2\nonumber\\
&&\quad\le e^{-\hat{\lambda}t-2\int_{-t}^0z(\theta_{s}\omega)ds}
\|B(\tau-t, \theta_{-t}\omega)\|^2\rightarrow 0 \ \mbox{as} \
t\rightarrow+\infty.
\end{eqnarray}
By \eqref{est2} and \eqref{tempered}, it follows that
\begin{eqnarray*}
&&\|v(\tau, \tau-t, \theta_{-\tau}\omega, v_{\tau-t})\|^2\\
&&\quad\quad\quad\quad \leq 1+(2|\| \alpha|\|^2+\frac{|\|
g|\|^2}{\lambda_0})\int_{-\infty}^0e^{\hat{\lambda}s
-2z(\theta_{s}\omega)+2\int_{s}^0z(\theta_{r}\omega)dr}ds.
\end{eqnarray*}
Now denoting
\begin{eqnarray}\label{50}
R^2(\omega)=1+(2|\| \alpha|\|^2+\frac{|\|
g|\|^2}{\lambda_0})\int_{-\infty}^0e^{\hat{\lambda}s
-2z(\theta_{s}\omega)+2\int_{s}^0z(\theta_{r}\omega)dr}ds,
\end{eqnarray}
we conclude that
\begin{eqnarray}\label{51}
\mathcal{K}(\tau, \omega)=\overline{B_{\ell^2}(0, R(\omega))}
\end{eqnarray}
is a closed measurable $\mathcal{D}(\ell^2)$-pullback absorbing set.
In fact, for all $\gamma>0$,
\begin{eqnarray*}
\begin{split}
e^{-\gamma t}R^2(\theta_{-t}\omega)&=e^{-\gamma t}+(2|\|
\alpha|\|^2+\frac{|\| g|\|^2}{\lambda_0})e^{-\gamma
t}\int_{-\infty}^0e^{\hat{\lambda}s
-2z(\theta_{s-t}\omega)+2\int_{s}^0z(\theta_{r-t}\omega)dr}ds\\
&=e^{-\gamma t}+(2|\| \alpha|\|^2+\frac{|\|
g|\|^2}{\lambda_0})e^{-\gamma
t}\int_{-\infty}^{-t}e^{\hat{\lambda}(s+t)
-2z(\theta_{s}\omega)+2\int_{s}^0z(\theta_{r}\omega)dr}ds\\
&\quad\quad\rightarrow 0 \ \ \mbox{as} \ \ t\rightarrow+\infty.
\end{split}
\end{eqnarray*}

\end{proof}

\subsection{Asymptotically Null of the Solutions}
In this subsection, the property of asymptotically null for the
solution $\Phi$ of system \eqref{5.9} will be established.

\begin{lemma} \label{Asymptotic Null}
Let $\mathcal{K}(\tau, \omega)$ be the absorbing set given by
\eqref{51}. Then for every $\epsilon>0$, there exist
$\tilde{T}(\epsilon, \tau, \omega, \mathcal{K}(\tau, \omega))>0$ and
$\tilde{N}(\epsilon, \tau, \omega, \mathcal{K}(\tau, \omega))\ge 1$,
such that the solution $\Phi(t, \tau-t, \theta_{-t}\omega)=v(\tau,
\tau-t, \theta_{-\tau}\omega)$ of problem \eqref{5.9} is
asymptotically null, that is, for all $t\geq \tilde{T}(\epsilon,
\tau, \omega, \mathcal{K}(\tau, \omega))$, $v_{\tau-t}\in B(\tau-t,
\theta_{-t}\omega)$,
\begin{eqnarray*}\sum_{|i|>\tilde{N}
(\epsilon, \tau, \omega, \mathcal{K}(\tau, \omega))}| (v(\tau,
\tau-t, \theta_{-\tau}\omega, v_{\tau-t}, g)_i|^2\leq
 \epsilon^2.
\end{eqnarray*}
\end{lemma}

\begin{proof}
Choose a smooth cut-off function satisfying $0\leq \rho(s)\leq 1$
for $s\in \mathbb{R^{+}}$ and $\rho(s)=0$ for $0\leq s\leq 1$,
$\rho(s)=1$ for $s\geq 2$. Suppose there exists a constant $c_0$
such that $|\rho'(s)|\leq c_0$ for $s\in \mathbb{R}^+$. For any
$n\ge 1$, let $v^{(m)}_n=v^{(m)}(\tau, \tau-t, \omega, v_{\tau-t,
n}, g_n)=(v^{(m)}_{n, i})_{i\in \mathbb{Z}}$ be a mild solution of
\eqref{int1}. Let $N$ be a fixed integer which will be specified
later, and set $x^{(m)}_{n}=(x^{(m)}_{n, i})_{i\in \mathbb{Z}}$
where $x^{(m)}_{n, i}=\rho(\frac{|i|}{N})v^{(m)}_{n, i}$ for any
$i\in \mathbb{Z}$. Then taking the inner product of \eqref{int1}
with $x$ in $\ell^2$, we obtain
\begin{equation}\label{4.4}
\begin{split}
&\frac{d}{dt}\sum_{i\in \mathbb{Z}}\rho(\frac{|i|}{N})|
v^{(m)}_{n, i}|^2\\
&\quad=-2\nu^{(m)}(t)(A_mv^{(m)}_n, x^{(m)}_{n})
-2(\lambda^{(m)}(t)-z(\theta_{t}\omega))\sum_{i\in \mathbb{Z}}
\rho(\frac{|i|}{N})|v^{(m)}_{n, i}|^2\\
&\quad\quad-2e^{-z(\theta_{t}\omega)}\sum_{i\in \mathbb{Z}}
\rho(\frac{|i|}{N})f(e^{z(\theta_{t}\omega)}v^{(m)}_{n, i}, t)
v^{(m)}_{n, i}\\
&\quad\quad+2e^{-z(\theta_{t}\omega)}(g_n(t), x^{(m)}_{n}).
\end{split}
\end{equation}
We now estimate terms in \eqref{4.4} one by one. First, we have
\begin{eqnarray}
(A_mv^{(m)}_n, x^{(m)}_{n})=(\tilde{B}_mv^{(m)},
\tilde{B}_mx^{(m)}_{n})\geq -\frac{2c_0}{N}\|v^{(m)}\|^2.
\label{4.5}
\end{eqnarray}
For the second term in \eqref{4.4}, it follows from the assumption
$(\mathbf{A}_2)$ that
\begin{eqnarray*}
-\infty<-2e^{-z(\theta_{t}\omega)} \sum_{i\in
\mathbb{Z}}\rho(\frac{|i|}{N})f_i(e^{z
(\theta_{t}\omega)}v^{(m)}_{n, i}, t) v^{(m)}_{n, i}\leq
2e^{-2z(\theta_{t}\omega)}\sum_{|i|\ge N}\alpha^2_i(t).
\end{eqnarray*}
For the last term in \eqref{4.4},
\begin{eqnarray}
2e^{-z(\theta_{t}\omega)}(g_n(t), x^{(m)}_{n})\leq
\lambda^{(m)}(t)\sum_{i\in \mathbb{Z}}\rho(\frac{|i|}{N})|
v_i|^2+\frac{1}{\lambda_0}e^{-2z(\theta_{t}\omega)}\sum_{|i|\geq
N}g^2_i(t). \label{4.6}
\end{eqnarray}
Combining \eqref{4.4}-\eqref{4.6}, it yields
\begin{equation}\label{4.7}
\begin{split}
&\frac{d}{dt}\sum_{i\in \mathbb{Z}}\rho(\frac{|i|}{N})| v^{(m)}_{n,
i}|^2+(\lambda_0-2z(\theta_{t}\omega))\sum_{i\in \mathbb{Z}}
\rho(\frac{|i|}{N})|v^{(m)}_{n, i}|^2\\
&\quad\le\frac{4\nu^0c_0}{N}\|v^{(m)}\|^2
+(2+\frac{1}{\lambda_0})e^{-2z(\theta_{t}\omega)}\sum_{|i|\ge
N}(\alpha^2_i(t)+g^2_i(t)).
\end{split}
\end{equation}
Apply Gronwall's inequality to \eqref{4.7} over $(\tau-t, \tau)$, we
obtain that
\begin{equation}\label{4.8}
\begin{split}
&\sum_{i\in \mathbb{Z}}\rho(\frac{|i|}{N})| v^{(m)}_{n, i}(\tau,
\tau-t, \omega,
v^{(m)}_{\tau-t, n}, g_n)|^2\\
&\le e^{-\lambda_0t-2\int_{\tau}^{\tau-t}
z(\theta_{r}\omega)dr}\|v^{(m)}_{\tau-t, n}\|^2\\
&\quad+\frac{4\nu^0 c_0}{N}\int_{\tau-t}^{\tau}
e^{-\lambda_0(\tau-s)+2\int_{s}^{\tau}z(\theta_{r}\omega)dr}
\|v^{(m)}(s, \tau-t, \omega,
v^{(m)}_{\tau-t, n}, g_n)\|^2ds\\
&\quad+(2+\frac{1}{\lambda_0})\sum_{|i|\geq
N}(\alpha^2_i(t)+g^2_i(t)) \int_{\tau-t}^{\tau}
e^{-\lambda_0(\tau-s)+2\int_{s}^{\tau}
z(\theta_{r}\omega)dr-2z(\theta_{s}\omega)}ds.
\end{split}
\end{equation}
Now, for $\tau\in \mathbb{R}$, substitute $\theta_{-\tau}\omega$ for
$\omega$ and estimate each term in \eqref{4.8}
\begin{equation}\label{4.9}
\begin{split}
&\sum_{i\in \mathbb{Z}}\rho(\frac{|i|}{N})| v^{(m)}_{n, i}(\tau,
\tau-t, \theta_{-\tau}\omega,
v^{(m)}_{\tau-t, n}, g_n)|^2\\
&\le e^{-\lambda_0t-2\int_{\tau}^{\tau-t}
z(\theta_{r-\tau}\omega)dr}\|v^{(m)}_{\tau-t, n}\|^2\\
&\quad+\frac{4\nu^0 c_0}{N}\int_{\tau-t}^{\tau}
e^{-\lambda_0(\tau-s)+2\int_{s}^{\tau}z(\theta_{r-\tau}\omega)dr}
\|v^{(m)}(s, \tau-t, \theta_{-\tau}\omega,
v^{(m)}_{\tau-t, n}, g_n)\|^2ds\\
&\quad+(2+\frac{1}{\lambda_0})\sum_{|i|\geq
N}(\alpha^2_i(t)+g^2_i(t)) \int_{\tau-t}^{\tau}
e^{-\lambda_0(\tau-s)+2\int_{s}^{\tau}
z(\theta_{r-\tau}\omega)dr-2z(\theta_{s-\tau}\omega)}ds\\
&\le e^{-\lambda_0t-2\int_{-t}^{0}
z(\theta_{r}\omega)dr}\|v^{(m)}_{\tau-t, n}\|^2\\
&\quad+\frac{4\nu^0 c_0}{N}\int_{-t}^{0}
e^{-\lambda_0s+2\int_{-t}^{0}z(\theta_{r}\omega)dr}
\|v^{(m)}(s+\tau, \tau-t, \theta_{-\tau}\omega,
v^{(m)}_{\tau-t, n}, g_n)\|^2ds\\
&\quad+(2+\frac{1}{\lambda_0})\sum_{|i|\geq
N}(\alpha^2_i(t)+g^2_i(t)) \int_{-t}^{0} e^{\lambda_0s+2\int_{s}^{0}
z(\theta_{r}\omega)dr-2z(\theta_{s}\omega)}ds.
\end{split}
\end{equation}
By Lemma \ref{abs set}, there exists $T_1(\epsilon, \tau, \omega,
\mathcal{K}(\omega))>0$ such that for all $t\ge T_1(\epsilon, \tau,
\omega, \mathcal{K}(\omega))$,
\begin{equation}\label{4.10}
\begin{split}
&\frac{4\nu^0 c_0}{N}\int_{-t}^{0}
e^{-\lambda_0s+2\int_{-t}^{0}z(\theta_{r}\omega)dr}
\|v^{(m)}(s+\tau, \tau-t, \theta_{-\tau}\omega,
v^{(m)}_{\tau-t, n}, g_n)\|^2ds\\
&\quad\quad\le \frac{4\nu^0 c_0}{\beta N}R^2(\omega),
\end{split}
\end{equation}
where $R^2(\omega)$ is given by \eqref{50}. Since $g(t),
\alpha(t)\in C_b(\mathbb{R}, \ell^2)$, by using \eqref{temperness}
again, we know
\begin{equation*}
\begin{split}
(2+\frac{1}{\lambda_0})\sum_{|i|\geq N}(\alpha^2_i(t)+g^2_i(t))
\int_{-\infty}^{0} e^{\lambda_0s+2\int_{s}^{0}
z(\theta_{r}\omega)dr-2z(\theta_{s}\omega)}ds<\infty,
\end{split}
\end{equation*}
and hence
\begin{equation}\label{4.11}
\begin{split}
\lim_{N\rightarrow\infty}(2+\frac{1}{\lambda_0})\sum_{|i|\geq
N}(\alpha^2_i(t)+g^2_i(t)) \int_{-\infty}^{0}
e^{\lambda_0s+2\int_{s}^{0}
z(\theta_{r}\omega)dr-2z(\theta_{s}\omega)}ds=0.
\end{split}
\end{equation}
Now, by means of \eqref{tempered} and \eqref{4.9}-\eqref{4.11},
there exist $\tilde{T}(\epsilon, \tau, \omega,
\mathcal{K}(\omega))\ge T_1(\epsilon, \tau, \omega,
\mathcal{K}(\omega))$ and $\tilde{N}(\epsilon, \tau, \omega,
\mathcal{K}(\omega))\ge 1$ such that
\begin{eqnarray}\label{4.15}
&&\sum_{|i|\geq \tilde{N}(\epsilon, \tau, \omega,
\mathcal{K}(\omega))}| v^{(m)}_{n, i}(\tau, \tau-t,
\theta_{-\tau}\omega,
v^{(m)}_{\tau-t, n}, g_n)|^2\nonumber\\
&\leq & \sum_{i\in \mathbb{Z}}\rho(\frac{|i|}{N})| v^{(m)}_{n,
i}(\tau, \tau-t, \theta_{-\tau}\omega, v^{(m)}_{\tau-t, n}, g_n)|^2
\leq \epsilon^2.
\end{eqnarray}

Since there is $m_k$ such that $$v^{(m_k)}_{n, i}(\tau, \tau-t,
\theta_{-\tau}\omega, v^{(m)}_{\tau-t, n}, g_n)\rightarrow (v(\tau,
\tau-t, \theta_{-\tau}\omega, v_{\tau-t, n}, g_n))_{i}$$ as
$m_k\rightarrow\infty$, by \eqref{4.15}
\begin{eqnarray*}
\sum_{|i|\geq \tilde{N}(\epsilon, \tau, \omega,
\mathcal{K}(\omega))}| (v(\tau, \tau-t, \theta_{-\tau}\omega,
v_{\tau-t, n}, g_n))_i|^2 \leq \epsilon^2
\end{eqnarray*}
for any $n\ge 1$. Now, letting $n\rightarrow\infty$ we can obtain
the conclusion.
\end{proof}

\subsection{Existence of Pullback Attractors}

We are now in a position to give our main result in this section.

\begin{theorem}
Suppose that $(\mathbf{A1}$-$\mathbf{A4})$ hold. The lattice
dynamical system $\Phi$ with both non-autonomous deterministic and
random forcing terms generated by system \eqref{5.9} has a unique
pullback attractor.
\end{theorem}

\begin{proof}
The result follows directly from Lemmas \ref{abs set},
\ref{Asymptotic Null} and Theorem \ref{existence 2}.
\end{proof}

\section*{Acknowledgments}
The first author thanks Prof. Bixiang Wang for emailing him the file
of reference \cite{Wang2}.



\end{document}